\documentclass[11pt]{amsart}
\usepackage{amsthm}
\usepackage{amsmath}
\usepackage{amssymb}
\usepackage{setspace}
\usepackage{fullpage}
\usepackage{color}
\usepackage{graphicx}
\usepackage{enumerate}
\usepackage[pdfstartview=FitH]{hyperref}
\hypersetup{
    pdftitle=   {Dynamic coloring of graphs having no K_5 minor},
   pdfauthor=  {Younjin Kim, Sang June Lee, Sang-il Oum}
}
\usepackage{tikz}
\usetikzlibrary{patterns}
\newtheorem{THM}{Theorem}

\newtheorem{LEM}[THM]{Lemma}
\newtheorem{DEF}[THM]{Definition}

\newtheorem{CLM}[THM]{Claim}
\newtheorem{PROP}[THM]{Proposition}

\newtheoremstyle{CASE}{}{}{\itshape}{}{\itshape}{:}{.5em}{#1 \thmnote{#3}}
\theoremstyle{CASE}

\theoremstyle{remark}

\def\tG{\widetilde{G}}

\def\deg{\mbox{deg}}

\newcommand\abs[1]{\lvert #1\rvert}

\begin{document}

\onehalfspace

\title{Dynamic coloring of graphs having no $K_5$ minor}
\author{Younjin Kim \and Sang June Lee \and Sang-il Oum}
\address[Younjin Kim, Sang-il Oum]{Department of Mathematical Sciences, Korea Advanced Institute of Science and Technology (KAIST), 291 Daehak-ro
  Yuseong-gu Daejeon, 305-701 South Korea} 
\address[Sang June Lee]{Department of Mathematics, Duksung Women's University, Seoul, 132-714 South Korea}
\email{younjin@kaist.ac.kr}
\email{sanglee242@duksung.ac.kr, sjlee242@gmail.com}
\email{sangil@kaist.edu}

\thanks{ 
The first author was supported by Basic Science Research Program through the
National Research Foundation of Korea(NRF) funded by the Ministry of
Science, ICT \& Future Planning(2013R1A1A3010982). \\ $\phantom{}$ 
The second author was supported by Basic Science Research Program through the National Research Foundation of Korea(NRF) funded by the Ministry of Science, ICT \& Future Planning (2013R1A1A1059913). \\ $\phantom{}$ The last author was supported by Basic Science Research Program through the National Research Foundation of Korea(NRF) funded by the Ministry of
Science, ICT \& Future Planning (2011-0011653). } 
\date{March 15, 2015}
\begin{abstract}
  We prove that every simple connected graph with no $K_5$ minor
  admits a proper $4$-coloring such that
  the neighborhood of each vertex $v$ having more than one neighbor is not monochromatic,
  unless the graph is isomorphic to the cycle of length $5$.
  This generalizes the result by S.-J.~Kim, S.~J.~Lee,  and
  W.-J.~Park{~\cite{KLP2012}} on planar graphs.
\end{abstract}
\maketitle

\section{introduction}
In this paper, all graphs are assumed to be simple, meaning that they
have no loops and no parallel edges.
We say that a vertex $v$ of a graph $G$ is \emph{dynamic} with 
a proper coloring of $G$,
if either $v$ has a pair of neighbors having distinct colors
or $v$ has at most one neighbor.
 A \emph{dynamic $k$-coloring} of a graph $G$ is a proper (vertex) $k$-coloring of
$G$ such that every vertex is dynamic.
The \emph{dynamic chromatic number} $\chi_d(G)$ of a graph $G$ is the
minimum number $k$ such that $G$ has a dynamic $k$-coloring.

This concept was introduced by Montgomery~\cite{Montgomery01}. 
Dynamic chromatic numbers (and list dynamic chromatic
numbers) have been studied for various classes of graphs such as graphs of small maximum degree,
bipartite graphs, regular graphs, random graphs, and graphs embedded in a
surface~\cite{Montgomery01, Lai03, Lai06, Meng06, Akbari09, Lin10, Alishahi11, Kim11, Chen12, Jahanbekan*, KLP2012}.

Erd\H{o}s, Furedi, Hajnal, Komj\'{a}th, R\"{o}dl, and
Seress~\cite{Erdos86} initiated a similar but opposite concept called
a local $\ell$-coloring. A \emph{local $\ell$-coloring} is a proper
coloring such that the neighbors of each vertex receive at most $l$
colors. There are series of results in this concept as well~\cite{Korner05, Simonyi06, Simonyi07, Gvozdenovic08, Omoomi08, Simonyi09, Kiviluoto09, Hajiabolhassan09, Alishahi10, Simonyi11}.

Clearly, for every graph $G$, the number $\chi_d(G)$ is at least the chromatic number $\chi(G)$.  It is easy to check that $\chi_d(C_5)=5$ and $\chi(C_5)=3$, and hence,  $\chi_d(G)$ may be strictly larger than $\chi(G)$. 
Moreover, a graph with small chromatic number may have arbitrarily large dynamic chromatic
number; for instance, if $G$ is the graph obtained from $K_n$ by subdividing every edge, then $\chi_d(G)=n$ but $\chi(G)=2$.
This might suggest that the dynamic chromatic
number of a graph may be quite different from the usual chromatic number. 
However, it turns out that every connected planar graph except for $C_5$ has a dynamic $4$-coloring, if we assume the four color theorem.
\begin{THM}[S.-J. Kim, S. J. Lee, W.-J. Park~\cite{KLP2012}]\label{thm:planar}
  If $G$ is a connected planar graph other than $C_5$, then $G$ is
  dynamically $4$-colorable.
\end{THM}

In this paper, we consider the dynamic chromatic numbers of graphs with no $K_5$
minor. Note that Wagner~\cite{Wagner37} proved that such graphs are
$4$-colorable, assuming the four color theorem. Our main theorem is as follows.

\begin{THM}\label{thm:main}
If $G$ is a connected graph other than $C_5$ having no $K_5$ minor, then $G$ is dynamically 4-colorable.
\end{THM}

Our proof is based on the structural characterization of graphs with
no $K_5$ minor by Wagner~\cite{Wagner37}. Unlike the usual graph
coloring, combining coloring on both side is not easy, because we
should make sure that every vertex is dynamic.

Next, we consider graphs with no $K_t$ minor for a general $t$ and show the following. 
\begin{THM}\label{thm:large t} For every integer $t\geq 2$, the following hold:
  \begin{enumerate}[(i)]
  \item A graph with no $K_{t}$
    topological minor is dynamically $(10t^2+2)$-colorable.
  \item A graph with no $K_{t}$
    minor is dynamically $(\lfloor 16t\sqrt{\log_2 t}\rfloor+3)$-colorable.
  \end{enumerate}
\end{THM}
See Section~\ref{sec:discussion} for the proof of Theorem~\ref{thm:large t} and the related discussion.

\vskip 1em
\noindent \textbf{Organization:} In Section~\ref{sec:proof}, we prove Theorem~\ref{thm:main}.
In the proof, we will use two properties of minimum counterexamples, Lemmas~\ref{lem:int3conn} and~\ref{lem:3cut}, that are proved in Sections~\ref{sec:int3conn} and~\ref{sec:4conn}, respectively. In Section~\ref{sec:discussion}, we discuss a related question
motivated by Hadwiger's conjecture and prove Theorem~\ref{thm:large t}.

\vskip 1em
\noindent \textbf{Notation:} Let $G$ be a graph. Let $V(G)$ denote the set of vertices of $G$. Let $E(G)$ denote the set of edges of $G$, and let $e(G)=|E(G)|$. For $u,v\in V(G)$, let $u\sim v$ denote that $u$ is adjacent to $v$. Let $N_G(v)$ denote the set of neighbors of $v$ in $G$. Let $\deg_{_{G}}(v)$ denote the degree of $v$ in $G$. For a vertex $v$ and a vertex set $U$, let $e(v,U)$ denote the number of edges between $v$ and vertices of $U$, that is, $e(v,U)=|N_G(v)\cap U|$.
For $x,y\in V(G)$, let $G+xy$ be the graph obtained from $G$ by adding the edge $xy$, and let $G/xy$ be the graph obtained from $G$ by identifying $x$ and $y$.


\section{Proof of Theorem~\ref{thm:main}}\label{sec:proof}

In order to prove Theorem~\ref{thm:main}, we will suppose that there is a counterexample of Theorem~\ref{thm:main}, and then, imply a contradiction. To this end, we show several properties of minimum counterexamples of Theorem~\ref{thm:main}. For the first property, we use the following definition.
\begin{DEF}
A graph $G$ with $\abs{V(G)}>3$  is called \emph{internally $3$-connected} if the following hold:
\begin{enumerate}[\hskip 2em (a)]
\item $G$ is $2$-connected,
\item For every separation $(A,B)$ of order $2$, we have that
$|A\setminus B|=1$ or $|B\setminus A|=1$. 
\end{enumerate}
\end{DEF}

\begin{LEM}\label{lem:int3conn}
  If $G$ is a counterexample of Theorem~\ref{thm:main} with minimum
number of edges,
  then $G$ is internally $3$-connected.
\end{LEM}

Our proof of Lemma~\ref{lem:int3conn} is given in Section~\ref{sec:int3conn}. The following lemma is another property of minimum counterexamples of Theorem~\ref{thm:main}.

\begin{LEM}\label{lem:3cut}
  Let $G$ be a counterexample of Theorem~\ref{thm:main} with minimum
number of edges.
  Then, for each $X\subset V(G)$ with $|X|=3$, we have that  $G\setminus X$ has
  at most $2$ components having a vertex of degree at least $3$.
\end{LEM}

Our proof of Lemma~\ref{lem:3cut} is provided in Section~\ref{sec:4conn}.

On the other hand, in order to prove Theorem~\ref{thm:main}, we use two known results: one is about graphs with no $K_5$ minor, and the other is about dynamic colorings.  Halin \cite{Halin1967,Halin1968} proved that every non-planar graph
with no $K_5$ minor contains a subdivision of $V_8$ (see Figure~\ref{fig:Wagner graph})  as a subgraph
or it has  a set $X$ of three vertices such that $G\setminus X$ has at
least three components. (A slightly stronger version was proved by
K\'ezdy and McGuinness~\cite[Theorem 3.6]{KM1992} later.) Moreover, Halin observed the following theorem since a $3$-connected graph not containing $K_5$ minor but containing $V_8$ minor  is isomorphic to $V_8$.

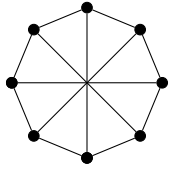
\begin{figure}
\begin{center}
  \tikzstyle{every node}=[circle,draw,fill=black!50,inner
  sep=0pt,minimum width=4pt]
  \begin{tikzpicture}
    \foreach \x in {0,2,4,6}
    \node [fill=black] (v\x) at (45*\x:1)  {} ;
    \foreach \x in {4,6}
    \node [fill=black] 
    (v\x) at (45*\x:1)  {} ;
    \foreach \x in {1,3}
    \node [fill=black] 
    (v\x) at (45*\x:1)  {} ;
    \foreach \x in {5,7}
    \node [fill=black] 
    (v\x) at (45*\x:1)  {} ;
    \draw (v0)--(v1)--(v2)--(v3)--(v4)--(v5)--(v6)--(v7)--(v0);
    \draw (v0)--(v4);
    \draw (v1)--(v5);
    \draw (v2)--(v6);
    \draw (v3)--(v7);
  \end{tikzpicture}

\end{center}\caption{The graph $V_8$.}\label{fig:Wagner graph}
\end{figure}

\begin{THM}[Halin \cite{Halin1967,Halin1968}]\label{thm:km}
  Every $3$-connected nonplanar graph with no $K_5$ minor
  is isomorphic to $V_8$ (see Figure~\ref{fig:Wagner graph})
  or has a set $X$ of three vertices such that $G\setminus X$ has at least
  $3$ components.
\end{THM}

We use the following theorem shown by Lai, Montgomery, and Poon~\cite{Lai03}. (A stronger
result with dynamic choice numbers was later proved by Akbari, Ghanbari, and Jahanbekam~\cite{Akbari09}.)
\begin{THM}[Lai, Montgomery, and Poon~\cite{Lai03}]
  \label{thm:degreebound} If a connected graph $H$ has maximum degree at most
  $3$, then $H$ is dynamically $4$-colorable, unless $H$ is isomorphic
  to $C_5$.
\end{THM}

We are ready to prove Theorem~\ref{thm:main}.
\begin{proof}[Proof of Theorem~\ref{thm:main}]
 Recall that $G$ is a counterexample of Theorem~\ref{thm:main} with minimum number of edges. Theorem~\ref{thm:planar} gives that $G$ is nonplanar.
  Let $\tG$ be the graph obtained from $G$ by
replacing each induced path of $G$ of length at least 2 with an edge.
Note that $G$ is a subdivision of $\tG$, and hence, $\tG$ is nonplanar and has no $K_5$ minor.

We claim that $\tG$ is $3$-connected. For a contradictory proof, suppose that $\tG$ has a separation $(A,B)$ with $A\cap B=\{x,y\}$ that separates vertices $u\in A$ and $v\in B$. Since $G$ is internally $3$-connected (see Lemma~\ref{lem:int3conn}), we infer that $A\setminus B=\{u\}$ or $B\setminus A=\{v\}$. Without loss of generality, let $A\setminus B=\{u\}$. Since $G$ is internally $3$-connected, and hence, $2$-connected, we have that $N_{\tG}(u)=\{x,y\}$, which contradicts to the fact that $\tG$ has no induced path of length $2$.

Therefore, Theorem~\ref{thm:km} gives that $\tG$ is isomorphic to $V_8$
  or has a set $X\subset V(\tG)$ of $3$ vertices such that $\tG\setminus X$ has at
  least $3$ components. If there is such a $X\subset V(\tG)$, there exists a set $X'\subset V(G)$ of $3$ vertices such that 
 $G\setminus X'$ has at least $3$ components having a vertex of degree at least $3$, which contradicts to Lemma~\ref{lem:3cut}.
  Therefore, $\tG$ is isomorphic to $V_8$, and hence, $G$ is a
  subdivision of $V_8$. 
Theorem~\ref{thm:degreebound} gives that $G$ is dynamically $4$-colorable. Consequently, $G$ is not a counterexample of Theorem~\ref{thm:main}.
\end{proof}


\section{Proof of Lemma~\ref{lem:int3conn}}\label{sec:int3conn}

A \emph{separation} of a graph is a pair $(A,B)$ of subsets of
$V(G)$ such that $A\cup B=V(G)$, $A\setminus B\neq\emptyset, B\setminus A\neq\emptyset$ and  there is no edge between $A\setminus
B$ and $B\setminus A$.
The \emph{order} of a separation $(A,B)$ is $\abs{A\cap B}$.

\begin{LEM}\label{lem:2conn}
  If a graph  $G$ is a counterexample of Theorem~\ref{thm:main} with minimum
number of edges, then $G$ is $2$-connected.
\end{LEM}
\begin{proof} 
Suppose, for contradiction, that $G$ is not $2$-connected. Then, $G$ has a separation $(A,B)$ of order $1$. 
Let  $A \cap B=\{x\}${\color{blue}.} 
If $G[A]$ and $G[B]$ are isomorphic to $C_5$, then $G$ is a connected planar graph other than $C_5$.
Theorem~\ref{thm:planar} gives that $G$ is dynamically $4$-colorable, and hence, $G$ is not a counterexample of Theorem~\ref{thm:main}.
 Therefore,  without loss of generality, we may assume that $G[B]$ is not isomorphic to $C_5$.

 Since $G[B]\subsetneq G$, we have that $e(G[B])<e(G)$ and $G[B]$ has no $K_5$ minor. From the assumption that $G$ is a minimum counterexample of Theorem~\ref{thm:main},  we have that $G[B]$ has a
dynamic $4$-coloring or is isomorphic to $C_5$. Recalling that $G[B]\neq C_5$, we have that $G[B]$ has a dynamic $4$-coloring, namely $c_2$.

 Similarly, $G[A]$ has a dynamic $4$-coloring or is isomorphic to $C_5$. If $G[A]$ has a dynamic $4$-coloring, let $c_1$ be such a coloring. On the other hand, if $G[A]$ is isomorphic to $C_5$, let $V(G[A])=\{x,a,b,c,d\}$ and let $x\sim a$, $a\sim b$, $b\sim c$, $c\sim d$, and $d\sim x$. Set $c_1(x)=1$, $c_1(a)=2$, $c_1(b)=3$, $c_1(c)=4$, and $c_1(d)=2$.
Notice that, in both cases, every vertex of $G[A]$ other than $x$ is
dynamic with $c_1$. 

By permuting colors of $c_2$, we let $c_1(x)=c_2(x)$.
Since $G$ is connected,  there are neighbors of $x$, namely $y$ and $z$, in $G[A]$ and $G[B]$, respectively.
We may assume that $$c_1(y)\neq c_2(z)$$ by permuting colors in $c_2$ without changing the color $c_2(x)$.
Let $c$ be the coloring of $G$ such that, for $u\in V(G)$,
\[
 c(u)=
 \begin{cases}
    c_1(u) & \text{if } u\in V(G[A]) \\
    c_2(u) & \text{otherwise.} 
  \end{cases}
\]

 The definition of $c_1$ and $c_2$ gives that
$c$ is a proper $4$-coloring of $G$. Next, each vertex $v\in A\setminus B$ is dynamic with $c_1$ in $G[A]$.  Since $N_{G[A]}(v)=N_G(v)$ and $c=c_1$ in $A$, we have that $v$ is dynamic with $c$ in $G$. Similarly, each vertex $u\in B\setminus A$ is dynamic with $c$ in $G$.
Next, recall that $x\in A\cap B$ has two neighbors $y\in V(G[A])$ and $z\in V(G[B])$, and  $$c(y)=c_1(y)\neq c_2(z)= c(z),$$ and hence, $x$ is dynamic with $c$ in $G$. Therefore, $c$ is a dynamic $4$-coloring of $G$, and hence, $G$ is not a counterexample of Theorem~\ref{thm:main}, which completes our proof of Lemma~\ref{lem:2conn}.
\end{proof}

Now, we prove Lemma~\ref{lem:int3conn}.

\begin{proof}[Proof of Lemma~\ref{lem:int3conn}] 
Suppose, for contradiction, that $G$ is not internally $3$-connected. Lemma~\ref{lem:2conn} gives that $G$ is $2$-connected.
Since $G$ is $2$-connected but not internally $3$-connected, there exists a separation
$(A,B)$ of order $2$ such that $\abs{A\setminus B}>1$ and $\abs{B\setminus
  A}>1$.
Let $A\cap B=\{x,y\}$.

Recalling that $G$ is a counterexample of Theorem~\ref{thm:main}, it follows from Theorem~\ref{thm:planar} that $G$ is nonplanar. Consequently, $G[A]+xy$ or $G[B]+xy$ is nonplanar.
Without loss of generality, we assume that $G[A]+xy$ is
nonplanar. Note that, since $G[A]+xy$ is nonplanar, $\abs{A}\ge 5$. 
We also assume that $\abs{A}$ is
minimized.   

Next, we claim that $\deg_{_{G[A]}}(x)\geq 2$ and  $\deg_{_{G[A]}}(y)\geq 2$.  Indeed, since $G$ is $2$-connected, we have that $x$ and $y$ have at least one neighbor in $A\setminus B$.
Suppose, for contradiction, that $x$ has only one neighbor $x'$ in $A\setminus B$. Then, $(A',B')=(A\setminus \{x\}, B\cup \{x'\})$ is a separation of $G$ of order $2$. 
We have that
$$ |A'\setminus B'|=|A\setminus B|-1\geq |A|-3\geq 2$$
and 
$$ |B'\setminus A'|=|B\setminus A|+1\geq 2.$$
Also,  
$G[A']+x'y$ is nonplanar because $G[A]+xy$ is nonplanar.
Therefore, $(A',B')$ is a separation with $\abs{A'\setminus B'}>1$ and $\abs{B'\setminus A'}>1$ satisfying that $G[A']+x'y$ is nonplanar, which contradicts to the minimality of $\abs{A}$.

We have the following claim. 
\begin{CLM}\label{clm:2sep}
  Let $G$ be a counterexample of Theorem~\ref{thm:main} with minimum
 number of edges.
 Suppose that $(A,B)$ with $A\cap B=\{x,y\}$ is a separation of $G$ of order $2$ satisfying that $\deg_{_{G[A]}}(x)\geq 2$, $\deg_{_{G[A]}}(y)\geq 2$, and $G[A]+xy$ is nonplanar.
 Then, the following hold:
  \begin{enumerate}[\hskip 2em (i)]
  \item There is a vertex $z\in B\setminus A$ with $N_G(z)=A\cap B$.
  \item $x\nsim y$.
  \item For every $w\in A\setminus B$, we have $N_G(w)\neq A\cap B$.
  \end{enumerate}
\end{CLM}  
Our proof of Claim~\ref{clm:2sep} will be given after the proof of Lemma~\ref{lem:int3conn}.
  
  Claim~\ref{clm:2sep} \textit{(i)} gives that there is a vertex $z \in B\setminus A$ with $N_G(z)=A\cap B$. Let $(A'',B'')=(A\cup\{z\},
B\setminus\{z\})$. Note that $$N_G(z)=A\cap B=\{x,y\}=A''\cap B''.$$ On the other hand,  we have that 
\begin{itemize}
\item $\deg_{_{G[A'']}}(x)\geq \deg_{_{G[A]}}(x)\geq 2$, and 
 $\deg_{_{G[A'']}}(y)\geq \deg_{_{G[A]}}(y)\geq 2$,
\item $G[A'']+xy$ is nonplanar, since $G[A'']\supset G[A]$ and $G[A]+xy$ is nonplanar.
\end{itemize} 
Hence, Claim~\ref{clm:2sep} \textit{(iii)} with $(A'',B'')$ implies that, for every $w\in A''\setminus B''$, we have $N_G(w)\neq A''\cap B''$, which contradicts to the fact that $N_G(z)= A''\cap B''$ with $z\in A''\setminus B''$.
\end{proof}

It remains to prove Claim~\ref{clm:2sep}.

\begin{proof}[Proof of Claim~\ref{clm:2sep}]

Suppose, for contradiction, that there is a separation $(A,B)$ of $G$ of order $2$ such that one of \textit{(i)}--\textit{(iii)} does not hold. We will infer that $G$ is dynamically $4$-colorable, which contradicts to the assumption that $G$ is a counterexample of Theorem~\ref{thm:main}.


Note that $G[A]\subsetneq G$, and recall that $G$ is a minimum counterexample of Theorem~\ref{thm:main}. Hence, $G[A]$ is isomorphic to $C_5$ or has a dynamic $4$-coloring. The assumption that $G[A]+xy$ is nonplanar implies that $G[A]$ is not isomorphic to $C_5$. Therefore, $G[A]$ has a dynamic $4$-coloring, namely $c_1$.

  We claim that $G[B]$ has a proper $4$-coloring $c_2$ such that
  $c_2(x)=c_1(x)$, $c_2(y)=c_1(y)$, and every vertex in $B\setminus A$
  is dynamic with $c_2$ in $G[B]$. 
If the claim holds, let $c$ be a coloring of $G$ defined by   \[
  c(v)=
  \begin{cases}
    c_1(v) & \text{if } v\in V(G[A]), \\
    c_2(v) & \text{otherwise.}
  \end{cases}
  \]
The properties of $c_1$ and $c_2$ imply that $c$ is a proper $4$-coloring and every vertex of $V(G)\setminus\{x,y\}$ is dynamic with $c$. Recalling the assumption $\deg_{_{G[A]}}(x)\geq 2$ and $\deg_{_{G[A]}}(y)\geq 2$, the fact that $x$ and $y$ are dynamic with $c_1$ in $G[A]$ implies that $x$ and $y$ are dynamic with $c$ in $G$. Therefore, $c$ is a dynamically $4$-coloring of $G$.
 
It only remains to show the above claim, that is, $G[B]$ has a proper $4$-coloring $c_2$ such that
  $c_2(x)=c_1(x)$, $c_2(y)=c_1(y)$, and every vertex in $B\setminus A$
  is dynamic with $c_2$ in $G[B]$. 
  We consider two cases (one is the case where $c_1(x)\neq c_1(y)$ and the other is the case where $c_1(x)=c_1(y)$) separately.  
   
  \begin{itemize}
  \item \textbf{Case 1}: 
    \emph{Suppose that $c_1(x)\neq c_1(y)$.}
Without loss of generality, let $c_1(x)=1$ and $c_1(y)=2$. Let $G[B]^+=G[B]+xy$ be the graph obtained from $G[B]$ by adding the edge $xy$ (if $x$ and $y$ are nonadjacent). The graph $G[B]^{+}$ is a minor of $G$, and hence, $G[B]^+$
  has no $K_5$ minor. 
  Also, $e(G[B]^+)<e(G)$. Recalling that $G$ is a minimum counterexample of Theorem~\ref{thm:main}, we infer that $G[B]^+$  has a dynamic $4$-coloring or is isomorphic to $C_5$.
  
If $G[B]^+$ has a dynamic $4$-coloring $c^+$,  set $c_2=c^+$ (as a coloring of $G[B]$).  Since $x\sim y$ in $G[B]^+$, we have that $c_2(x)\neq c_2(y)$, and hence, by permuting colors of $c_2$, we have that $c_2(x)=1$ and $c_2(y)=2$. On the other hand, if $G[B]^{+}$ is isomorphic to $C_5$, let $x,a,b,c,y$ be the vertices of $G[B]^{+}$ in the cyclic order. Set $c_2(x)=1$, $c_2(a)=2$, $c_2(b)=3$, $c_2(c)=4$, and $c_2(y)=2$  (as a coloring of $G[B]$).

One can easily see that $c_2$ is a proper $4$-coloring with $c_2(x)=1=c_1(x)$ and $c_2(y)=2=c_1(y)$.
 Next, each vertex of $B\setminus A$ is dynamic with $c_2$ in $G[B]^+$.
 Since each vertex $u\in B\setminus A$ satisfies $N_{G[B]^+}(u)=N_{G[B]}(u)$, each vertex of $B\setminus A$ is dynamic with $c_2$ in $G[B]$. 
  
\vskip 1em

\item \textbf{Case 2}:
\emph{Suppose that $c_1(x)=c_1(y)$.}
Without loss of generality, we may assume that $c_1(x)=c_1(y)=1$. Let $G[B]^*=(G[B]+xy)/xy$ denote the graph obtained from $G[B]$ by identifying $x$ and $y$.
The graph $G[B]^*$ is a minor of $G$, and hence, $G[B]^*$
  has no $K_5$ minor. 
  Also, $e(G[B]^*)<e(G)$. Recalling that $G$ is a minimum counterexample of Theorem~\ref{thm:main}, we infer that $G[B]^*$  has a dynamic $4$-coloring or is isomorphic to $C_5$.

If $G[B]^*$ has a dynamic $4$-coloring $c^*$,  set $c_2=c^*$ (as a coloring of $G[B]$).  Since $x=y$ in $G[B]^*$, we have that $c_2(x)=c_2(y)$, and hence, by permuting colors of $c_2$, we have that $c_2(x)=c_2(y)=1$. On the other hand, if $G[B]^*$ is isomorphic to $C_5$, 
let $x=y,a,b,c,d$ be the vertices of $G[B]^*$ in the cyclic order. Set $c_2(x)=c_2(y)=1$, $c_2(a)=2$, $c_2(b)=3$, $c_2(c)=4$, and $c_2(d)=2$ (as a coloring of $G[B]$).

We have that $c_1(x)=c_1(y)=c_2(x)=c_2(y)$.
Next, since $c_1(x)=c_1(y)$ and $c_1$ is proper in $G[A]$, we infer that $x\nsim y$ in $G[A]$, and hence, in $G[B]$. Hence, $c_2$ is a proper $4$-coloring of $G[B]$.

Next, we claim that every vertex in $B\setminus A$ is dynamic with $c_2$
even in $G[B]$ instead of $G[B]^*$.
Indeed, let $v\in B\setminus A$ be such that $\deg_{_{G[B]}}(v)\geq 2$. If $\deg_{_{G[B]^*}}(v)=1$, we have that $N_{G}(v)=\{x,y\}$. Since \textit{(ii)} or \textit{(iii)} does not hold, we infer that $c_1(x)\neq c_1(y)$, which implies a contradiction. Therefore, $\deg_{G[B]^*}(v)\geq 2$. Consequently, $v$ has two neighbors with distinct colors of $c_2$. Thus, $v\in B\setminus A$ is dynamic with $c_2$ in $G[B]$.
\end{itemize}

Cases 1 and 2 complete our proof of Claim~\ref{clm:2sep}.
\end{proof}


\section{Proof of Lemma~\ref{lem:3cut}}\label{sec:4conn}


\begin{proof}[Proof of Lemma~\ref{lem:3cut}]
	Suppose, for contradiction, that $G$ has a set $X=\{x_1,x_2,x_3\}$ such that $G\setminus X$ has
	  at least $3$ components (namely $C_1$, $C_2$, and $C_3$) having  a vertex of degree at least $3$. One of the followings holds:
\begin{itemize}
\item $|C_1|=|C_2|=|C_3|=1$. 
\item $|C_1|\geq 2$ without loss of generality.
\end{itemize}
 Set $C=C_1$, and let $A=V(G)\setminus V(C)$ and $B=X\cup V(C)$. Note that $A\cap B=X$.
 
Since $G\setminus X$ has
	  at least $2$ components, other than $C$, having a vertex of degree at least $3$, we have that $A\setminus B$ has at least $2$ vertices of degree at least $3$.
 
 Next, we claim that, for $i=1,2,3$, $$\deg_{_{G[A]}}(x_i)\geq 2.$$ 
 Recall that, for a vertex $v$ and a vertex set $U$, let $e(v,U)=|N_G(v)\cap U|$.   For a proof of the claim, it suffices to show that, for $i=1,2,3$ and $j=2,3$, 
 \begin{equation}\label{eq:x,C}e(x_i,C_j)\geq 1.\end{equation} We consider the case where  $j=2$.
\begin{itemize}
\item \textit{Suppose that $|C_2|=1$, that is, $C_2=\{v\}$.} Recalling that $C_2$ has a vertex of degree at least $3$, we have $\deg_{_{G}}(v)\geq 3$, and hence, $N_G(v)=\{x,y,z\}$. Thus,~\eqref{eq:x,C} holds for $i=1,2,3$ and $j=2$. 
\item \textit{Suppose that $|C_2|\geq 2$.} Suppose that there is an $i$ such that $e(x_i,C_2)=0$. Without loss of generality, let $e(x_1,C_2)=0$. Set $A'=V(G)\setminus V(C_2)$ and $B'=\{x_2,x_3\}\cup V(C_2)$. Then, $(A',B')$ is a separation of order $2$ satisfying that $|A'\setminus B'|\geq 2$ and $|B'\setminus A'|\geq 2$. This contradicts to the fact that $G$ is  internally $3$-connected (see Lemma~\ref{lem:int3conn}).
\end{itemize} Therefore,~\eqref{eq:x,C} with $j=2$ holds. Symmetrically, ~\eqref{eq:x,C} with $j=3$ holds.

  
  We have the following claim.
  \begin{CLM}\label{clm:3sep}
  Let $G$ be a counterexample of Theorem~\ref{thm:main} with minimum
number of edges.
Suppose that $(A,B)$ with $A\cap B=\{x,y,z\}$ is a separation of $G$ of order $3$ satisfying the following:
\begin{enumerate}[\hskip 2em (a)]
  \item $A\setminus B$ has at least $2$ vertices of degree at least $3$. 
   \item $\deg_{_{G[A]}}(x)\geq 2$, $\deg_{_{G[A]}}(y)\geq 2$, and $\deg_{_{G[A]}}(z)\geq 2$.  
\end{enumerate}

  Then, the following hold: 
  \begin{enumerate}[\hskip 2em (i)]
  \item There is a vertex $u\in B\setminus A$ with $N_G(u)\subset A\cap B$, 
    \item There is no edge in $N_G(u)$.
    \item  For every $w\in A\setminus B$, we have $N_G(w)\neq N_G(u)$.
  \end{enumerate}
\end{CLM}
 
 Our proof of Claim~\ref{clm:3sep} will be given after the proof of Lemma~\ref{lem:3cut}.

Note that the separation $(A,B)$ with $A=V(G)\setminus V(C)$ and $B=X\cup V(C)$ satisfies \textit{(a)} and \textit{(b)} of Claim~\ref{clm:3sep}, and hence, \textit{(i)--(iii)} of Claim~\ref{clm:3sep} hold.
Suppose that $|C|\geq 2$. Then, \textit{(i)} of Claim~\ref{clm:3sep} gives that there is a vertex $u\in C=B\setminus A$ satisfying $N_G(u)\subset X=A\cap B$. Then, $\{u\}\subsetneq V(C)$ is a component of $G\setminus X$, which contradicts to the fact that $C$ is a component with $|C|\geq 2$. Hence, we may assume that $|C|=1$.

From the definition of $C$, we have that $|C_1|=|C_2|=|C_3|=1.$ Let $C=C_1=\{v_1\}$ and $C_2=\{v_2\}$. Since $\deg_{_{G}}(v_1)\geq 3$, we infer that $N_G(v_1)=X$. Similarly, since $\deg_{_{G}}(v_2)\geq 3$, we also have that $N_G(v_2)=X$. This contradicts to \textit{(iii)} of Claim~\ref{clm:3sep}, which completes our proof of Lemma~\ref{lem:3cut}.
\end{proof}

It remains to prove Claim~\ref{clm:3sep}.

\begin{proof}[Proof of Claim~\ref{clm:3sep}]
Suppose, for contradiction, that $G$ has a separation $(A,B)$ with $A\cap B=\{x,y,z\}$ such that one of \textit{(i)--(iii)} does not hold.
We will infer that $G$ is dynamically $4$-colorable, and hence, $G$ is not a counterexample of Theorem~\ref{thm:main}. 

First, we claim that $G[B]+xy+yz+zx$ has no $K_5$ minor. Indeed,
 Lemma~\ref{lem:int3conn} gives that $G$ is internally $3$-connected, and hence, $G$ is $2$-connected. Therefore, for each vertex $v\in A\setminus B$, we have $\deg_{_{G[A]}}(v)\geq 2$. This together with the assumption \textit{(b)} yields that, for each vertex $v\in A$, we have $\deg_{_{G[A]}}(v)\geq 2$. Therefore, $G[A]$ has a cycle, namely $C$.
Since  $G$ is internally $3$-connected, there are $3$ vertex-disjoint paths from
  $A\cap B$ to $V(C)$. (Here, a vertex in $(A\cap B)\cap V(C)$ is counted as a path connecting $A\cap B$ to $V(C)$.) We obtain
  $K_3$ on $A\cap B$ by contracting edges in those paths and the cycle $C$. 
  This implies that 
 $G[B]+xy+yz+zx$
  is a minor of $G$, and hence, it has no $K_5$ minor.

The assumption \textit{(a)} gives that $G[A]$ is not isomorphic to $C_5$. Therefore, by the
  minimality assumption, $G[A]$ has a dynamic $4$-coloring $c_1$.

We claim that $G[B]$ has a proper $4$-coloring $c_2$ such that 
  $c_1(x)=c_2(x)$, $c_1(y)=c_2(y)$, $c_1(z)=c_2(z)$,
  and every vertex in $B\setminus A$ is dynamic with $c_2$  in $G[B]$,
  If this claim is true, then we have
  a proper $4$-coloring $c$ of $G$ as
  \[
  c(t)=
  \begin{cases}
    c_1(t) & \text{if } t\in V(G[A]) \\
    c_2(t) & \text{otherwise.}
  \end{cases}
  \]
 One may check that $c$ is a dynamic $4$-coloring of $G$ by the assumption \textit{(b)}.
  This contradicts to our assumption that $G$ is a counterexample,
  finishing the proof.
  
  To prove the claim, we consider the following three cases separately. 
\begin{itemize}
\item \textbf{Case 1}:
    \textit{Suppose that $c_1(x)$, $c_1(y)$, and $c_1(z)$ are distinct.
Recall that $G[B]+xy+yz+zx$ has no $K_5$ minor and we have $e(G[B]+xy+yz+zx)< e(G)$.   Since $G[B]+xy+yz+zx$ is not isomorphic to $C_5$, 
  it admits a dynamic $4$-coloring $c_2$.
   By permuting colors of $c_2$, we may let} $c_2(x)=c_1(x)$, $c_2(y)=c_1(y)$, $c_2(z)=c_1(z)$.
  Clearly, every vertex in $B\setminus A$ is dynamic with $c_2$  in $G[B]+xy+yz+zx$, and hence, in $G[B]$.
\vskip 1em

\item \textbf{Case 2}:
    \textit{Suppose that $c_1(x)=c_1(y)\neq c_1(z)$.}
  Set $G'=(G[B]+xy+yz)/xy$. Note that $G'$ has no $K_5$
  minor  and $e(G')<e(G)$. If $G'$ is not isomorphic to $C_5$, then we take a dynamic
  $4$-coloring $c_2$ of $G'$  to define a proper $4$-coloring of
  $G[B]$.  On the other hand, if $G'$ is isomorphic to $C_5$, 
  then let $x=y,a,b,c,z$ be
  the vertices of $G'$ in the cyclic order. We take $c_2(x)=c_2(y)=1,
  c_2(a)=2, c_2(b)=4, c_2(c)=3$, and $c_2(z)=2$.
  By permuting colors of $c_2$, we have that   $c_1(x)=c_2(x)$, $c_1(y)=c_2(y)$, and $c_1(z)=c_2(z)$.
  
We claim that every vertex in $B\setminus A$ is dynamic
  with $c_2$ even in $G[B]$ instead of $G'$. Indeed, let $v\in B\setminus A$ be such that $\deg_{_{G[B]}}(v)\geq 2$. If $\deg_{_{G'}}(v)=1$, then $N_{G[B]}(v)=\{x,y\}$. Since \textit{(ii)} or \textit{(iii)} does not hold, we infer that $c_1(x)\neq c_1(y)$. Therefore, $\deg_{_{G'}}(v)\geq 2$. Thus, $v$ has two neighbors with distinct coloring in $c_2$, and hence, $v$ is dynamic with $c_2$ in $G[B]$.

\vskip 1em
\item \textbf{Case 3}:
    \textit{Suppose that $c_1(x)=c_1(y)=c_1(z)$.}
Set $G''=(G[B]+xy+yz)/xy/yz$. Since $G''$ is a minor of $G$, the graph $G''$ has no $K_5$
  minor. Also, $e(G'')<e(G)$. If $G''$ is not isomorphic to $C_5$, then we take a dynamic
  $4$-coloring $c_2$ of $G''$ to define a proper $4$-coloring of
  $G[B]$. 
On the other hand, if $G''$ is isomorphic to $C_5$, 
  then let $x=y=z,a,b,c,d$ be the vertices of $G''$ in the cyclic
  order. 
  We take $c_2(x)=c_2(y)=c_2(z)=1, c_2(a)=2, c_2(b)=3, c_2(c)=4, c_2(d)=2$. 
 By permuting colors of $c_2$, we have that $c_1(x)=c_2(x)$, $c_1(y)=c_2(y)$, and $c_1(z)=c_2(z)$.

We claim that every vertex in $B\setminus A$ is dynamic with $c_2$ even in $G[B]$ instead of $G''$. Indeed, let $v\in B\setminus A$ be such that $\deg_{_{G[B]}}(v)\geq 2$. If $\deg_{_{G''}}(v)=1$, then $N_{G[B]}(v)\subset \{x,y,z\}$. 
Since \textit{(ii)} or \textit{(iii)} does not hold, we infer that $c_1(x)=c_1(y)=c_1(z)$ is not possible. Therefore, $\deg_{_{G''}}(v)\geq 2$. Thus, $v$ has two neighbors with distinct coloring in $c_2$, and hence, $v$ is dynamic with $c_2$ in $G[B]$.

   \end{itemize}
  
  Cases 1--3 completes our proof of Claim~\ref{clm:3sep}.
\end{proof}


\section{Graphs with no $K_t$ minor for a general $t$}
\label{sec:discussion}
Hadwiger's conjecture \cite{Hadwiger1943}
claims that graphs with no $K_ t$ minor is
$ (t-1)$-colorable. 
It is fairly easy to prove Hadwiger's conjecture for $ t\leq 4$.
The four color theorem is equivalent to Hadwiger's conjecture for $t=5$, shown by
Wagner~\cite{Wagner37},
and for $t=6$ shown by Robertson, Seymour, and Thomas~\cite{RST1993a}.
Hadwiger's conjecture remains open for all $t\ge 7$.

Our theorem on dynamic coloring can also be seen as an equivalent
theorem of the four color theorem, as there is only one exception
$C_5$ which is $3$-colorable. So we might ask the following:
\begin{quote}
  Is it true that every connected graph with no $K_t$ minor is dynamically
  $(t-1)$-colorable, except finitely many $(t-1)$-colorable graphs?
\end{quote}
This question is obviously true for $t=2$, 
and is true for $t=5$ by Theorem~\ref{thm:main}.

But unlike Hadwieger's conjecture, 
this question  turns out to be false when $t=3$ or $t=4$.
Obviously there are no dynamically $2$-colorable connected graphs having a
vertex of degree at least $2$ and therefore the question is false for
$t=3$.
For $t=4$, every cycle of length $3k\pm 1$ for a positive integer $k$ is
not dynamically $3$-colorable \cite{Montgomery01}.
Moreover, for any dynamically $3$-colorable graph, we can attach an
ear of length $3k$ whose ends are adjacent to forbid it to be
dynamically $3$-colorable. This suggests that probably it is not
easy to characterize exceptional graphs for $t=4$.

It remains to consider the question for $ t\geq 6$; the authors are
unaware of any graph having no $K_{t}$ minor but not dynamically
$(t-1)$-colorable for $t\geq 6$.
Probably with more colors available, it may be easier to color the
graph dynamically.

As a small evidence, let us consider apex graphs or more generally,
graphs that can be made planar by removing at most $k$ vertices.
If $G\setminus X$ is planar
for a set $X$ of $k$ vertices, then $G$ has no $K_{k+5}$ minor.
We prove that those graphs are dynamically $(k+4)$-colorable.
\begin{PROP}
  Let $k$ be a positive integer.
  If $G\setminus X$ is planar for a set $X$ of $k$ vertices, 
  then $G$ is dynamically $(k+4)$-colorable.
\end{PROP}
\begin{proof}
  We proceed by induction on $\abs{V(G)}+k$.
  
  Let $v\in X$. 
  We may assume that $G\setminus (X\setminus \{v\})$ is nonplanar and
  therefore $v$ has at least two neighbors.

  If $G\setminus v$ is isomorphic to $C_5$, then $G$ is planar and
  therefore $G$ is dynamically $4$-colorable.
  Thus we may assume that $G\setminus v$ is not isomorphic to $C_5$.
  By the induction hypothesis, $G\setminus v$ is dynamically
  $(k+3)$-colorable.
  Let $c_1$ be a dynamic $(k+3)$-coloring of $G\setminus v$.

  Let $N$ be the set of all neighbors of $v$.
  If $N$ is not monochromatic in $c_1$, then 
  we win because $c$  is  a dynamic $(k+4)$-coloring of $G$.
  Thus we may assume that 
  every vertex in $N$ is colored by $1$ in $c_1$.
  Let us pick a neighbor $w$ in $N$.

Suppose that every neighbor of $v$ has at least three neighbors in $G$.
  Let 
  \[
  c(x)=
  \begin{cases}
    2  & \text{if }x=v,\\
    k+4 & \text{if }x=w,\\
    c_1(x) & \text{if }x\in V(G)\setminus\{v,w\}.
  \end{cases}
  \]
  We claim that $c$ is a dynamic $(k+4)$-coloring of $G$. 
  It is easy to see that $c$ is a proper coloring.
  Moreover, $v$ is dynamic with $c$ because its neighbors have color
  $1$ and $k+4$ in $c$.
  A vertex $y\neq v,w$ is dynamic with $c$
  because it was dynamic with $c_1$.
  A neighbor $x$ of $v$ other than $w$ is dynamic with $c_1$
  because it has a neighbor of color $2$
  and a neighbor of color other than $2$.
  The vertex $w$ is dynamic with $c_1$ because not all neighbors of $w$
  other than $v$ are colored with $2$ due to the assumption that $c_1$
  is a dynamic $(k+3)$-coloring of $G\setminus v$.
  This proves the claim. 

Now we may assume that $v$ has a neighbor $w$ having at most two neighbors in $G$.
   \begin{itemize}
  \item   If $w$ has a neighbor $w_1$ other than $v$
    and $w_1$ has  a neighbor $w_2$ other than $w$, 
    then let $d\in\{2,\ldots,k+3\}\setminus\{c_1(w_1), c_1(w_2)\}$.
  \item If $w$ has  a neighbor $w_1$ other than $v$ and $w_1$ has no
    neighbor other than $w$, 
    then let $d\in \{2,\ldots,k+3\}\setminus\{c_1(w_1)\}$.
  \item If $w$ has no neighbor other than $v$, then 
    let $d=2$.
  \end{itemize}
  Let 
  \[
  c(x)=
  \begin{cases}
    k+4 & \text{if }x=v,\\
    d & \text{if }x=w,\\
    c_1(x) & \text{if } x\in V(G)\setminus\{v,w\}.
  \end{cases}
  \]
  Then it is easy to check that $c$ is a dynamic $(k+4)$-coloring of $G$.
\end{proof}

We also prove that 
a graph with no
$K_t$ minor is dynamically $f(t)$-colorable for some function $f$.
To show this, we prove the following lemma.
A graph $H$ is a \emph{topological minor} of $G$
if $G$ has a subgraph isomorphic to a subdivision of $H$.
\begin{LEM}\label{lem:degenerate}
  For a positive integer $k$, let $G$ be a graph such that every topological minor of $G$ has a vertex
  of degree at most $k$.
  Then $G$ is dynamically $(k+3)$-colorable.
\end{LEM}

Note that unlike the usual coloring of graphs, $k$-degenerate graphs
have unbounded dynamic chromatic number. 
(A graph is \emph{$k$-degenerate} if every subgraph has a vertex of
degree at most $k$.)
For instance, if $G$ is a
graph obtained from $K_n$ by subdividing each edge once, then
$G$ is $2$-degenerate and yet
$\chi_d(G)\ge n$.

\begin{proof}
  We may assume that $G$ is simple connected and $k\ge 2$.
  We proceed by induction on $\abs{V(G)}$.
  We may assume that $\abs{V(G)}\ge 2$.
  Let $v$ be a vertex of minimum degree.
  Let $v_1,v_2,\ldots,v_d$ be neighbors of $v$.

  If $v$ has degree $1$, then 
  we can pick a color for $v$ that are not used in $v_1$
  and one of the neighbors of $v_1$.
 So we may assume that every vertex of $G$ has degree at least $2$.
  Let $w_1$ be a neighbor of $v_1$ in $G\setminus v$.
  Let $w_2$ be a neighbor of $v_2$ in $G\setminus v$.

  If $v$ has degree $2$, then 
  let $G_1=G/vv_1$.
  Let $c_1$ be a dynamic $(k+3)$-coloring of $G_1$. Note that $G_1$ is a
  topological minor of $G$ and therefore $G_1$ satisfies our
  assumption that each of its topological minors has a vertex of
  degree at most $k$.  Since $v_1$ and $v_2$ are adjacent,
  $c_1(v_1)\neq c_1(v_2)$. 
  Let $f$ be a color in $\{1,2,\ldots,k+3\}\setminus \{c_1(v_1),
  c_1(v_2),c_1(w_1), c_1(w_2)\}$. Since $k\ge 2$, it is possible to
  choose $f$. Let 
  \[
  c(x)=
  \begin{cases}
    f &\text{if }x=v,\\
    c_1(x)&\text{if }x\in V(G)\setminus\{v\}.
  \end{cases}
  \]
  Then $c$ is a dynamic $(k+3)$-coloring of $G$.

  We may now assume that every vertex of $G$ has degree at least $3$.
  Let $G_1=(G\setminus v) + v_1v_2$. Notice that $G_1$ is a topological
  minor of $G$ as we can remove all edges incident with $v$ except
  $vv_1$ and $vv_2$ and then contract $vv_1$.
  By the induction hypothesis, $c_1$ is dynamically $(k+3)$-colorable.
  Let $f$ be a color in $\{1,2,\ldots,k+3\}\setminus (\{c_1(v_1), c_1(v_2),\ldots,c_1(v_d),c_1(w_1), c_1(w_2)\})$. Since $k+3\ge d+2$, it is possible to
  choose $f$. Let 
  \[
  c(x)=
  \begin{cases}
    f &\text{if }x=v,\\
    c_1(x)&\text{if }x\in V(G)\setminus\{v\}.
  \end{cases}
  \]
  Then $c$ is a dynamic $(k+3)$-coloring of $G$. First of all, $v$ is
  dynamic with $c$ because $v_1$ and $v_2$ are adjacent in $G_1$
  and so $c_1(v_1)\neq c_1(v_2)$.
  Secondly, for $i=1,2$, $v_i$ is dynamic with $c$ because $w_i$ and $v$
  are neighbors of $v_i$ having distinct colors.
  Third, for $i=3,4,\ldots,d$, the vertex $v_i$ is dynamic with $c$
  because $v_i$ has degree at least $2$ in $G_1$ and therefore $v_i$
  has a pair of neighbors having distinct colors in $c_1$.
  Finally, all other vertices are dynamic with $c$ because they were
  dynamic with $c_1$ in $G_1$.
\end{proof}

We also use known results in order to show our result. Bollob\'as and Thomason \cite{BT1998} and Koml\'os and Szemer\'edi~\cite{KS1996b}
  proved that there is an absolute constant $c$ such that, for every $t\geq 2$, each graph having
  average degree at
  least $ct^2$ contains $K_t$ as a topological minor. A theorem by Thomas and
  Wollan~\cite{TW2005}  implies that $c$ can be taken to be $10$,
  explained in the textbook of Diestel~\cite{Diestel2010}.
  Thus, we have that, for every $t\geq 2$,
  \begin{equation}\label{eq:topological minor} 
  \mbox{each graph with no $K_{t}$ topological minor has a vertex of
  degree less than $ 10t^2 $.}
  \end{equation}

 There are similar results for minors. Mader~\cite{Mader1968} proved that if an $n$-vertex graph has at
  least $(8t\sqrt{\log_2 t})n$ edges, then it has $K_t$ minor.
  This gives that, for every $t\geq 2$,
  \begin{equation}\label{eq:minor} \mbox{each graph with no $K_t$ minor has a vertex of degree
  less than $16t\sqrt{\log_2 t}$.} \end{equation}
 (Later Thomason~\cite{Thomason2001} showed that 
  this constant $8$ can be improved to $0.2656\ldots+o_t(1)$ where $o_t(1)\rightarrow 0$ as $t\rightarrow \infty$. Note that this
  constant is asymptotically tight.)

We are ready to show Theorem~\ref{thm:large t}.

\begin{proof}[Proof of Theorem~\ref{thm:large t}]
  \textit{(i)}: Let $G$ be a graph with no $K_t$ topological minor. Clearly, every topological minor of $G$ has no $K_t$ topological minor. From~\eqref{eq:topological minor}, every topological minor of $G$ has a vertex of degree less than $10t^2$. Hence, Lemma~\ref{lem:degenerate} implies \textit{(i)}.
  
  \textit{(ii)}: Let $G$ be a graph with no $K_t$ minor. Clearly, every minor of $G$ has no $K_t$ minor. In particular, every topological minor of $G$ has no $K_t$ minor. From~\eqref{eq:minor}, every topological minor of $G$ has a vertex of degree less than $16t\sqrt{\log_2 t}$. Hence, Lemma~\ref{lem:degenerate} yields \textit{(ii)}.
 \end{proof}


\begin{thebibliography}{10}

\bibitem{Akbari09}
S.~Akbari, M.~Ghanbari, and S.~Jahanbekam.
\newblock On the list dynamic coloring of graphs.
\newblock {\em Discrete Appl. Math.}, 157(14):3005--3007, 2009.

\bibitem{Alishahi11}
M.~Alishahi.
\newblock On the dynamic coloring of graphs.
\newblock {\em Discrete Appl. Math.}, 159(2-3):152--156, 2011.

\bibitem{Alishahi10}
M.~Alishahi, H.~Hajiabolhassan, and A.~Taherkhani.
\newblock A generalization of the {E}rd{\H o}s-{K}o-{R}ado theorem.
\newblock {\em Discrete Math.}, 310(1):188--191, 2010.

\bibitem{BT1998}
B.~Bollob{\'a}s and A.~Thomason.
\newblock Proof of a conjecture of {M}ader, {E}rd{\H o}s and {H}ajnal on
  topological complete subgraphs.
\newblock {\em European J. Combin.}, 19(8):883--887, 1998.

\bibitem{Chen12}
Y.~Chen, S.~Fan, H.-J. Lai, H.~Song, and L.~Sun.
\newblock Conditional coloring for planar graphs and graphs of higher genus.
\newblock {\em Discrete Applied Mathematics}, 160:1064--1071, 2012.

\bibitem{Diestel2010}
R.~Diestel.
\newblock {\em Graph theory}, volume 173 of {\em Graduate Texts in
  Mathematics}.
\newblock Springer, Heidelberg, fourth edition, 2010.

\bibitem{Erdos86}
P.~Erd{\H{o}}s, Z.~F{\"u}redi, A.~Hajnal, P.~Komj{\'a}th, V.~R{\"o}dl, and
  {\'A}.~Seress.
\newblock Coloring graphs with locally few colors.
\newblock {\em Discrete Math.}, 59(1-2):21--34, 1986.

\bibitem{Gvozdenovic08}
N.~Gvozdenovi{\'c} and M.~Laurent.
\newblock The operator {$\Psi$} for the chromatic number of a graph.
\newblock {\em SIAM J. Optim.}, 19(2):572--591, 2008.

\bibitem{Hadwiger1943}
H.~Hadwiger.
\newblock \"{U}ber eine {K}lassifikation der {S}treckenkomplexe.
\newblock {\em Vierteljschr. Naturforsch. Ges. Z\"urich}, 88:133--142, 1943.

\bibitem{Hajiabolhassan09}
H.~Hajiabolhassan.
\newblock On colorings of graph powers.
\newblock {\em Discrete Math.}, 309(13):4299--4305, 2009.

\bibitem{Halin1967}
R.~Halin.
\newblock Zur {K}lassifikation der endlichen {G}raphen nach {H}. {H}adwiger und
  {K}. {W}agner.
\newblock {\em Math. Ann.}, 172:46--78, 1967.

\bibitem{Halin1968}
R.~Halin.
\newblock On the classification of finite graphs according to {H}. {H}adwiger
  and {K}. {W}agner.
\newblock In {\em Theory of {G}raphs ({P}roc. {C}olloq., {T}ihany, 1966)},
  pages 161--167. Academic Press, New York, 1968.

\bibitem{Jahanbekan*}
S.~Jahanbekan, J.~Kim, S.~O, and D.~West.
\newblock $r$-dynamic coloring of graphs.
\newblock Submitted.

\bibitem{KM1992}
A.~K{\'e}zdy and P.~McGuinness.
\newblock Sequential and parallel algorithms to find a {$K_5$} minor.
\newblock In {\em Proceedings of the {T}hird {A}nnual {ACM}-{SIAM} {S}ymposium
  on {D}iscrete {A}lgorithms ({O}rlando, {FL}, 1992)}, pages 345--356, New
  York, 1992. ACM.

\bibitem{KLP2012}
S.-J. Kim, S.~J. Lee, and W.-J. Park.
\newblock Dynamic coloring and list dynamic coloring of planar graphs.
\newblock {\em Discrete Applied Math.}, 161:2207--2212, 2013.

\bibitem{Kim11}
S.-J. Kim and W.-J. Park.
\newblock List dynamic coloring of sparse graphs.
\newblock In {\em Combinatorial optimization and applications}, volume 6831 of
  {\em Lecture Notes in Comput. Sci.}, pages 156--162. Springer, Heidelberg,
  2011.

\bibitem{Kiviluoto09}
L.~Kiviluoto, P.~R.~J. {\"O}sterg{\aa}rd, and V.~P. Vaskelainen.
\newblock Sperner capacity of small digraphs.
\newblock {\em Adv. Math. Commun.}, 3(2):125--133, 2009.

\bibitem{KS1996b}
J.~Koml{\'o}s and E.~Szemer{\'e}di.
\newblock Topological cliques in graphs. {II}.
\newblock {\em Combin. Probab. Comput.}, 5(1):79--90, 1996.

\bibitem{Korner05}
J.~K{\"o}rner, C.~Pilotto, and G.~Simonyi.
\newblock Local chromatic number and {S}perner capacity.
\newblock {\em J. Combin. Theory Ser. B}, 95(1):101--117, 2005.

\bibitem{Lai06}
H.-J. Lai, J.~Lin, B.~Montgomery, T.~Shui, and S.~Fan.
\newblock Conditional colorings of graphs.
\newblock {\em Discrete Math.}, 306(16):1997--2004, 2006.

\bibitem{Lai03}
H.-J. Lai, B.~Montgomery, and H.~Poon.
\newblock Upper bounds of dynamic chromatic number.
\newblock {\em Ars Combin.}, 68:193--201, 2003.

\bibitem{Lin10}
Y.~Lin and K.~W. Zhao.
\newblock Dynamic coloring of planar graphs.
\newblock {\em J. Zhengzhou Univ. Nat. Sci. Ed.}, 42(3):34--36, 2010.

\bibitem{Mader1968}
W.~Mader.
\newblock Homomorphies\"atze f\"ur {G}raphen.
\newblock {\em Math. Ann.}, 178:154--168, 1968.

\bibitem{Meng06}
X.~Meng, L.~Miao, B.~Su, and R.~Li.
\newblock The dynamic coloring numbers of pseudo-{H}alin graphs.
\newblock {\em Ars Combin.}, 79:3--9, 2006.

\bibitem{Montgomery01}
B.~Montgomery.
\newblock {\em Dynamic coloring of graphs}.
\newblock ProQuest LLC, Ann Arbor, MI, 2001.
\newblock Thesis (Ph.D.)--West Virginia University.

\bibitem{Omoomi08}
B.~Omoomi and A.~Pourmiri.
\newblock Local coloring of {K}neser graphs.
\newblock {\em Discrete Math.}, 308(24):5922--5927, 2008.

\bibitem{RST1993a}
N.~Robertson, P.~Seymour, and R.~Thomas.
\newblock Hadwiger's conjecture for {$K_6$}-free graphs.
\newblock {\em Combinatorica}, 13(3):279--361, 1993.

\bibitem{Simonyi06}
G.~Simonyi and G.~Tardos.
\newblock Local chromatic number, {K}y {F}an's theorem and circular colorings.
\newblock {\em Combinatorica}, 26(5):587--626, 2006.

\bibitem{Simonyi07}
G.~Simonyi and G.~Tardos.
\newblock Colorful subgraphs in {K}neser-like graphs.
\newblock {\em European J. Combin.}, 28(8):2188--2200, 2007.

\bibitem{Simonyi11}
G.~Simonyi and G.~Tardos.
\newblock On directed local chromatic number, shift graphs, and {B}orsuk-like
  graphs.
\newblock {\em J. Graph Theory}, 66(1):65--82, 2011.

\bibitem{Simonyi09}
G.~Simonyi, G.~Tardos, and S.~T. Vre{\'c}ica.
\newblock Local chromatic number and distinguishing the strength of topological
  obstructions.
\newblock {\em Trans. Amer. Math. Soc.}, 361(2):889--908, 2009.

\bibitem{TW2005}
R.~Thomas and P.~Wollan.
\newblock An improved linear edge bound for graph linkages.
\newblock {\em European J. Combin.}, 26(3-4):309--324, 2005.

\bibitem{Thomason2001}
A.~Thomason.
\newblock The extremal function for complete minors.
\newblock {\em J. Combin. Theory Ser. B}, 81(2):318--338, 2001.

\bibitem{Wagner37}
K.~Wagner.
\newblock \"{U}ber eine {E}igenschaft der ebenen {K}omplexe.
\newblock {\em Math. Ann.}, 114(1):570--590, 1937.

\end{thebibliography}
 \end{document}